\documentclass[11pt,english]{article}
\usepackage[margin= 2 cm,bottom=23mm, top= 20mm]{geometry}

\usepackage{amsthm}
\usepackage{amsmath}
\usepackage{amssymb}
\usepackage{setspace}
\usepackage{mathtools}
\usepackage{verbatim}
\usepackage{csquotes}
\usepackage[hidelinks]{hyperref}
\usepackage{thm-restate}
\usepackage{cleveref}
\usepackage{enumitem}
\setlist[itemize]{leftmargin=*} 
\setlist[enumerate]{leftmargin=*}
\usepackage{framed}
\usepackage{floatrow}
\usepackage[T1]{fontenc}
\usepackage{bbm}
\usepackage[color=Orange!50!white, textwidth=22mm]{todonotes}

\usepackage{mathdots}
\usepackage{xcolor}
\usepackage{diagbox}
\usepackage{colortbl}

\usepackage{graphicx}
\usepackage{subcaption}

\theoremstyle{plain}

\newtheorem{theorem}{Theorem}[section]
\newtheorem{claim}[theorem]{Claim}

\newtheorem{lemma}[theorem]{Lemma}

\theoremstyle{definition}

\newtheorem*{defn*}{Definition}

\expandafter\def\expandafter\normalsize\expandafter{%
    \normalsize
    \setlength\abovedisplayskip{4pt}
    \setlength\belowdisplayskip{4pt}
    \setlength\abovedisplayshortskip{4pt}
    \setlength\belowdisplayshortskip{4pt}
}

\usepackage[square,sort,comma,numbers]{natbib}
\newcommand{\calF}{\mathcal{F}}

\newcommand{\calG}{\mathcal{G}}

\DeclareMathOperator{\E}{\mathbb{E}}

\newcommand{\Bin}{\mathrm{Bin}}

\newcommand{\unq}{\subseteq^*}
\newcommand{\unqembed}{\rightarrow^*}
\renewcommand{\Pr}{\mathbb{P}}

\title{Unique subgraphs are rare}
\author{Domagoj Brada\v{c}\thanks{Department of Mathematics, ETH Z\"urich, Z\"urich, Switzerland. Research supported in part by SNSF grant 200021-228014. Email: \textbf{domagoj.bradac@math.ethz.ch}}\and
Micha Christoph\thanks{Department of Mathematics, ETH Z\"urich, Z\"urich, Switzerland. Research supported in part by SNSF Ambizione grant 216071. Email: \textbf{micha.christoph@math.ethz.ch}}
}  

\date{}

\begin{document}

\maketitle
\begin{abstract}
    A folklore result attributed to P\'olya states that there are $(1 + o(1))2^{\binom{n}{2}}/n!$ non-isomorphic graphs on $n$ vertices. Given two graphs $G$ and $H$, we say that $G$ is a unique subgraph of $H$ if $H$ contains exactly one subgraph isomorphic to $G$. For an $n$-vertex graph $H$, let $f(H)$ be the number of non-isomorphic unique subgraphs of $H$ divided by $2^{\binom{n}{2}}/n!$ and let $f(n)$ denote the maximum of $f(H)$ over all graphs $H$ on $n$ vertices. In 1975, Erd{\H{o}}s asked whether there exists $\delta>0$ such that $f(n)>\delta$ for all $n$ and offered $\$100$ for a proof and $\$25$ for a disproof, indicating he does not believe this to be true. We verify Erd\H{o}s' intuition by showing that $f(n)\rightarrow 0$ as $n$ tends to infinity, i.e. no graph on $n$ vertices contains a constant proportion of all graphs on $n$ vertices as unique subgraphs.
\end{abstract}
\section{Introduction}

Given a family of graphs $\mathcal{F},$ we say that a graph $H$ is (induced) universal for $\mathcal{F}$ if it contains every graph in $\calF$ as an (induced) subgraph. The study of universal graphs was initiated by Rado~\cite{rado1964universal} who rediscovered an old construction of Ackermann~\cite{ackermann1937widerspruchsfreiheit} of a countable graph which contains all countable graphs as induced subgraphs. Since then, the focus has been on finite families $\mathcal{F}$ and constructing small universal graphs, either with few vertices or with few edges.

Alon~\cite{alon2017asymptotically} constructed a graph on $(1+o(1))2^{(k-1)/2}$ vertices which is induced universal for the family of all graphs on $k$ vertices and a simple counting argument shows this is tight up to an $(1+o(1))$-factor. Most of the remaining work has focused on (not induced) universal graphs for families of sparse graphs such as graphs with bounded maximum degree~\cite{Alon2002SparseUG, Alon2007SparseUG, Alon2008OptimalUG, Alon2000UniversalityAT}, forests~\cite{Chung1983OnUG, Friedman1987ExpandingGC, Graham1978OnGW} and graphs with bounded degeneracy~\cite{allen2023universality, nenadov2016ramsey}. Many of these results have recently been unified by Alon, Dodson, Jackson, McCarty, Nenadov and Southern~\cite{alonuniversality} and we refer to their paper and the references therein for more details about the history of universal graphs.

It should be mentioned that for many natural families $\calF$, a tight lower bound on the size of the universal graph is given by simple counting arguments. Moreover, these arguments often imply that if a graph is significantly smaller than the given lower bound, then in fact it contains a $o(1)$-proportion of the members of $\calF$ as (induced) subgraphs.

We move on to defining the problem which is the main topic of this paper. An asymptotic formula for the number of non-isomoprhic graphs on $n$ vertices is a folklore result often attributed to P\'olya. It seems that P\'olya himself never explicitly wrote this down but see, for example, the argument of Wright \cite{wright1971graphs} from 1971 for an estimate of the number of unlabelled graphs on $n$ vertices with exactly $m$ edges, which also implies the following.
\begin{theorem}[P\'olya; Wright] \label{thm:polya}
    The number of non-isomorphic graphs on $n$ vertices is $(1 + o(1))\frac{2^{\binom{n}{2}}}{n!}.$    
\end{theorem}

In 1972, Entringer and Erd{\H{o}}s \cite{entringer1972number} studied the following universality type question. We say that a graph $G$ is a unique subgraph of $H$, denoted by $G \unq H$, if $H$ contains exactly one subgraph isomorphic to $G$, where we remind the reader that a subgraph of a graph is defined as a subset of its vertices and a subset of its edges on that set of vertices. For an $n$-vertex graph $H$, we define 
\[ f(H) := \frac{|\{ G \, \vert \, G \unq H\}|}{2^{\binom{n}{2}} / n!}, \]
and let $f(n)$ be the maximum of $f(H)$ over all graphs $H$ on $n$ vertices. 
Entringer and Erd{\H{o}}s showed that there exists a constant $c$ such that for sufficiently large $n$, $f(n)\geq e^{-cn^{3/2}}$. Harary and Schwenk \cite{harary1973number} subsequently improved the lower bound to $e^{-cn\log n}$. The current best known bound is $f(n)\geq e^{-cn}$, given by Brouwer \cite{brouwer1975number}. In a paper from 1975, Erd{\H{o}}s \cite{erdHos1975problems} asked whether there exists $\delta>0$ such that $f(n)>\delta$ for all $n$ and offered $\$100$ for a proof and $\$25$ for a disproof. The problem also appears as Problem \#426 on Bloom's Erd\H{o}s problems website~\cite{bloom}. We answer this question in the negative.
\begin{theorem}\label{thm: main}
    \[ \lim_{n \rightarrow \infty} f(n) = 0. \]
\end{theorem}

One way to view \Cref{thm: main} is as an anti-concentration result. Using Theorem~\ref{thm:polya}, it is not hard to show that $f(H)$ is up to an additive $o(1)$ term equal to the probability that a random graph $G \sim \calG(n, 1/2)$ is a unique subgraph of $H$. Given a graph $H$ on $n$ vertices and $G\sim\calG(n,1/2)$, the expected number of embeddings of $G$ into $H$ equals $n!\cdot 2^{-\binom{n}{2}+e(H)}$. If $H$ has the right number of edges then this expectation can be close to $1$. Nonetheless, as our result shows, the probability that there is exactly one copy of $G$ is $o(1)$. Furthermore, a key ingredient of our proof is anti-concentration of the number of edges of $\calG(n, 1/2)$.

This interpretation of our result seems quite similar to the study of the parameter $\mathrm{ind}(k,\ell)$ introduced by Alon, Hefetz, Krivelevich and Tyomkyn \cite{alon2020edge}, where $\mathrm{ind}(k,\ell)$ is defined as the limit, as $n$ goes to infinity, of the maximum fraction of induced subgraphs on $k$ vertices which also have exactly $\ell$ edges over all graphs on $n$ vertices. This parameter has a natural similarity to $f$. Again, one may pick a graph $H$ on $n$ vertices with the appropriate number of edges such that the expected number of edges in any induced subgraph on $k$ vertices is $\ell$. If $0 < \ell < \binom{k}{2}$, anti-concentration comes into play. Confirming conjectures from~\cite{alon2020edge}, in a series of papers, Kwan, Sudakov and Tran \cite{kwan2019anticoncentration}, Fox and Sauermann \cite{fox2018completion} and Martinsson, Mousset, Noever and Truji{\'c} \cite{martinsson2019edge} showed that for all other values of $\ell$, $\mathrm{ind}(k,\ell)$ essentially does not exceed $1/e$ and most recently, with their work on the quadratic Littlewood-Offord problem, Kwan and Sauermann \cite{kwan2023resolution} showed that $\mathrm{ind}(k,\ell)\leq O(1/\sqrt{k})$ for some range of $\ell$, which is asymptotically tight.


\textbf{Notation, terminology and preliminaries.} We use mostly standard notation. For a positive integer $n$, we denote by $\calG(n, 1/2)$, the binomial random graph with edge probability $1/2$. Equivalently, this is a uniformly random labelled graph on $n$ vertices. For $0\leq m\leq \binom{n}{2}$, we denote by $\calG(n,m)$ the uniform distribution on labelled graphs with $n$ vertices and $m$ edges. Additionally, denoting $N = \binom{n}{2}$, the random graph process $\widetilde{G} = (G_m)_{m=0}^N$ is a sequence of labelled graphs $G_0 \subseteq G_1 \subseteq \dots \subseteq G_N$ on $n$ vertices, where $G_0$ is the empty graph and, for $1 \le m \le N,$ the graph $G_m$ is obtained from $G_{m-1}$ by adding a uniformly random edge not in $G_{m-1}$. Note that $|E(G_m)| = m$ and that $G_m$ is distributed as $\calG_{n,m}$.
We denote by $\mathrm{Aut}(G)$ the automorphism group of $G$. Note that for every graph $|\mathrm{Aut}(G)| \geq 1$, since $G$ always has the trivial automorphism. Given two graphs $G$ and $H$, an embedding of $G$ into $H$ is an injective function $\phi \colon V(G) \rightarrow V(H)$ such that $uv \in E(G) \implies \phi(u)\phi(v) \in E(H)$. Similar to the use of $\unq$, we write $G \unqembed H$ if there exists a unique embedding of $G$ into $H$. Throughout the paper, $n$ will denote the number of vertices in a graph $H$ and all our asymptotic notation is for $n$ tending to infinity. For clarity, we avoid the use of floor and ceiling signs whenever they are not crucial to the argument.

\subsection{Proof sketch}
For the sake of contradiction, we assume there is a constant $\delta > 0$ such that $f(n) \ge \delta$ for all $n$. We take a large enough $n$ compared to $\delta$ and consider a graph $H$ on $n$ vertices with $f(H)\geq \delta$.

The proof is split into three steps, formulated in \Cref{lem:can-look-at-random-graph,lem:constant-non-degree,lem: linear is not it}. 

In the proof of \Cref{thm: main}, we have to treat graphs $G$ with more than one automorphism delicately, as, for such $G$, it is possible that $G$ has more than one embedding into $H$ but, nonetheless, $G$ is a unique subgraph of $H$. Using \Cref{thm:polya} and a simple counting argument, Lemma~\ref{lem:can-look-at-random-graph} shows that $f(H)$ is, up to an additive $o(1)$ term, equal to $\Pr_{G \sim \calG(n,1/2)}[G \unqembed H]$. This allows us to look at $G \sim \calG(n,1/2)$ (the uniform distribution on labelled graphs) instead of the uniform distribution on unlabelled graphs on at most $n$ vertices as well as to ignore the delicate case when $G$ has non-trivial automorphisms.

In \Cref{lem:constant-non-degree}, we show that $e(H)\geq \binom{n}{2}-Cn$ for some large enough constant $C = C(\delta)$. Towards this, we observe that $e(G)$, where $G\sim \calG(n,1/2)$, with high probability lies in an interval of size linear in $n$ around $\binom{n}{2}/2$, while being more or less uniformly distributed within this interval. Shortly switching to the random graph process instead of $\calG(n,1/2)$, we use the previous observation and \Cref{lem:can-look-at-random-graph} to conclude that there must exist a graph $G^*$ with $G^*\unqembed H$ and if we, uniformly at random, add $\Omega(n)$ many missing edges to $G^*$ then it still has a unique embedding into $H$ with constant probability. A simple calculation then shows that this implies that $H$ has $O(n)$ non-edges.

Finally, \Cref{lem: linear is not it} states that if $H$ has $O(n)$ non-edges, then the probability that there is a unique embedding of $G\sim\calG(n,1/2)$ into $H$ is $o(1)$. Therefore, this concludes the proof together with \Cref{lem:can-look-at-random-graph}. Note that the expected number of embeddings of $G$ into $H$ is $n! \cdot 2^{\binom{n}{2} - e(H)} \gg 1$ for $H$ with $O(n)$ non-edges. Thus, at first glance, it seems we want to show that with high probability, there are at least two embeddings of $G$ into $H$. However, this need not be the case, for example if $H$ has an isolated vertex. In some sense, the vertices of large non-degree in $H$ are the only obstruction possibly preventing $G$ to have many embeddings into $H$ with high probability. We show that if there is an embedding of $G$ into $H$, then with high probability there are two vertices in $H$ with $O(1)$ non-neighbors whose roles we can switch to obtain a new embedding of $G$. Our proof uses a union bound over all $n!$ initial embeddings and Azuma's inequality to show that for any fixed initial embedding, there exists a way to switch two vertices with probability $1 - o(1/n!)$.


\section{Proof of Theorem~\ref{thm: main}}
\subsection{Reduction to random labelled graphs}
\begin{lemma} \label{lem:can-look-at-random-graph}
    Let $H$ be an $n$-vertex graph. Then, 
    \[ f(H) = \Pr_{G \sim \calG(n,1/2)}[G \unqembed H ] + o(1). \]
\end{lemma}
\begin{proof}
    Let $UG(n)$ denote the set of unlabelled graphs on $n$ vertices and let $LG(n)$ denote the set of all labelled graphs on $n$ vertices. Note that
    \[ 2^{\binom{n}{2}} = |LG(n)| = \sum_{G \in UG(n)} \frac{n!}{|\mathrm{Aut}(G)|} \le n! \cdot |UG(n)| - n!/2 \cdot \big|\{G \in UG(n) \, \vert \, |\mathrm{Aut}(G)| \ge 2 \}\big|. \]

    Using Theorem~\ref{thm:polya} and dividing by $n!$ yields
    \begin{equation} \label{eq:almost-all-have-one-auto}
        \big|\{G \in UG(n) \, \vert \, |\mathrm{Aut}(G)| \ge 2\}\big| = o\left(  \frac{2^{\binom{n}{2}}}{n!} \right).
    \end{equation}
    
    Note that by Theorem~\ref{thm:polya}, the number of unlabelled graphs on at most $n-1$ vertices is $o\left(\frac{2^{\binom{n}{2}}}{n!} \right)$. Using this and~\eqref{eq:almost-all-have-one-auto}, we have

    \begin{align*}
        f(H) &= \frac{n!}{2^{\binom{n}{2}}} \cdot \big|\{G \in UG(n) \, \vert \, G \unq H\}\big| + o(1) 
        \\&= \frac{n!}{2^{\binom{n}{2}}} \cdot 
        \big|\{G \in UG(n) \, \vert \, G \unq H \land |\mathrm{Aut}(G)| = 1\}\big|  +  o(1)\\
        &= 2^{-\binom{n}{2}} \cdot \big| \{ G \in LG(n) \, \vert \, G \unq H \land |\mathrm{Aut}(G)| = 1 \} \big| + o(1) \\&= 
        \Pr_{G \in \calG(n,1/2)}[G \unq H \land |\mathrm{Aut}(G)| = 1] + o(1). \hspace{.3cm}
    \end{align*}
    To conclude the proof, observe that $G\unqembed H$ if and only if $G\unq H$ and $|\mathrm{Aut}(G)| = 1$.
\end{proof}
\subsection{Reduction to the very dense case}
\begin{lemma} \label{lem:constant-non-degree}
    For any $\delta > 0$, there is a constant $C = C(\delta)$ such that the following holds. If $H$ is a graph on $n$ vertices satisfying $\Pr_{G \sim \calG(n,1/2)}[G \unqembed H ]\geq \delta/2$, then $e(H) \ge \binom{n}{2} - Cn$.
\end{lemma}
\begin{proof}
    Fix $\delta > 0$. Clearly, we may assume $n$ is sufficiently large by increasing $C$. Let $G \sim \calG(n, 1/2)$ and let $H$ be an $n$-vertex graph with $\Pr[G \unqembed H ]\geq \delta/2$. 
    By the Chernoff bound (e.g.~\cite{janson-luczak-ruczinski}), there is a constant $L = L(\delta)$ such that $\Pr[| e(G) - \binom{n}{2} / 2| \ge Ln] \leq \delta / 4.$ Let $I = [\binom{n}{2}/2 - Ln, \binom{n}{2} / 2 + Ln].$ Then,
    \[
        \Pr[G \unqembed H \land e(G) \in I] \ge \delta / 4.
    \]

    Set $N = \binom{n}{2}$ and note that by Stirling's approximation, for large enough $n$ and any $t$, \[ \Pr[\Bin(N,1/2) = t] \le \Pr[\Bin(N,1/2) = N/2] \le 1/\sqrt{N}. \] Therefore, $\Pr[e(G) = m] \le 2 / n$ for all $m$. Consider the random graph process $\widetilde{G} = (G_m)_{m=0}^N$ independent of $G$ and let
    \[ X = | \{ m \in I \, \vert \, G_m \unqembed H\}|. \]
    Note that
    \begin{align*}        
        \delta/4 &\le \Pr\big[G \unqembed H \land e(G) \in I\big] = \sum_{m \in I} \Pr\big[e(G) = m\big] \cdot \Pr\big[G_m \unqembed H \big]\\
        &\le \frac{2}{n} \sum_{m \in I} \Pr\big[G_m \unqembed H\big] = \frac{2}{n} \E \big[ X\big].
    \end{align*}
    Therefore $\E[X] \ge \delta n / 8.$ Since $X \le |I| \le 3Ln,$ we have 
    \[ \Pr[X \ge \delta n / 16] \ge \frac{\delta}{48L}. \]
    If $X \ge 1,$ let (the random variable) $m^*$ denote the minimum value of $m \in I$ such that $G_m \unqembed H$. By the law of total probability, there exists a graph $G^*$ such that
    \[ \Pr_{\widetilde{G}} \left[ X \ge \delta n / 16 \, \vert \, G_{m^*} = G^*\right] \ge \frac{\delta}{48L}. \]

    Let $m_2 = m^* + \delta n / 16-1$. Observe that the number of embeddings of $G_m$ into $H$ is non-increasing in $m$. Consequently, the set of $m$ such that $G_m\unqembed H$ forms an interval. Therefore, if $G_{m^*} = G^*$ and $X \ge \delta n /16$ then $G_{m_2} \unqembed H$.
    It follows that    
    \[ \Pr\left[G_{m_2} \unqembed H \, \vert \, G_{m^*} = G^* \right] \ge \Pr\left[X \ge \delta n /16 \, \vert \, G_{m^*} = G^* \right] \ge \frac{\delta}{48L}. \]

    Let $\pi$ be the unique embedding from $G^*$ into $H$. Since $G_{m^*} \subseteq G_{m_2},$ if $G_{m_2} \unqembed H$, then it has the same unique embedding $\pi$. Note that, given $G_{m^*} = G^*,$ the graph $G_{m_2}$ is uniformly distributed among all supergraphs of $G^*$ with $m_2$ edges. Therefore,
    \begin{align*}
        \Pr\big[G_{m_2} \unqembed H \, \vert \, G_{m^*} = G^*\big] = \frac{\binom{e(H) - m^*}{m_2 - m^*}}{\binom{N-m^*}{m_2-m^*}} \le \left( \frac{e(H) - m^*}{N - m^*} \right)^{m_2 - m^*} \le \left( \frac{e(H)}{N} \right)^{\delta n /16-1}.
    \end{align*}

    Assume that $e(H) \le N - Cn$ for large enough $C = C(\delta)$. Then,
    \[ \Pr\big[G_{m_2} \unqembed H \, \vert \, G_{m^*} = G^*\big] \le (1 - Cn / N)^{\delta n / 17} \le e^{-\frac{Cn^2 \delta}{17N}} < \frac{\delta}{48L}, \]
    where we took $C$ to be large enough, a contradiction. Hence, $e(H) \ge \binom{n}{2} - Cn$, as claimed.    
\end{proof}
\subsection{Very Dense $H$}
We use Azuma's inequality in the following form. 
\begin{theorem}[e.g. \cite{janson-luczak-ruczinski}]\label{thm: azumas}
    Let $Z_1,\ldots, Z_m$ be independent random variables, with $Z_i$ taking values in a set $\Lambda_i$. Assume that $f:\ \Lambda_1\times\ldots\times\Lambda_m\rightarrow \mathbb{R}$ satisfies, for some values $b_i$, $i=1,\ldots,m$, the following condition:
    \begin{enumerate}[label=(\Alph*)]
        \item If two vectors $z,z'\in \Lambda_1\times\ldots\times\Lambda_m$ differ only in the $i$th coordinate, then $|f(z)-f(z')|\leq b_i$.
    \end{enumerate}
    Then, the random variable $X=f(Z_1.\ldots,Z_m)$ satisfies, for any $t\geq 0$,
    \begin{align*}
        \Pr[X\leq \mathbb{E}[X]-t]\leq e^{\frac{-2t^2}{\sum_{i=1}^m b_i^2}}.
    \end{align*}
\end{theorem}
\begin{lemma}\label{lem: linear is not it}
    Let $C \ge 0$ be a constant. Then, for any graph $H$ on $n$ vertices with $e(H)\geq \binom{n}{2}-Cn$, it holds that
    \[\Pr_{G \sim \calG(n,1/2)}[G \unqembed H ] = o(1).\]
\end{lemma}
\begin{proof}
    Note that we may simply assume that $C\geq 1$. Let $H^c$ denote the complement of $H$. Since for $G\sim \mathcal{G}(n,1/2)$, it holds that the complement of $G$ follows the same distribution, the statement is equivalent to showing that with probability $o(1)$, for $G\sim \mathcal{G}(n,1/2)$, $H^c\unqembed G$ (where the graph $G$ here is the complement of the graph $G$ in the statement). Let $\pi$ be a bijection from $V(H^c)$ to $V(G)$. For distinct $u,v\in V(G)$, we say that the unordered pair $\{u,v\}$ is a \emph{$\pi$-switch} if    
    \[ N_G(v) \supseteq \pi(N_{H^c}(\pi^{-1}(u))\setminus N_{H^c}(\pi^{-1}(v))) \text{ and } N_G(u) \supseteq \pi(N_{H^c}(\pi^{-1}(v))\setminus N_{H^c}(\pi^{-1}(u))). \]
    Observe that if $\pi$ is an embedding of $H^c$ into $G$ and $\{u,v\}$ is a $\pi$-switch, then there is another embedding $\pi'$ of $H^c$ into $G$ obtained by switching the roles of $u$ and $v$. Formally, $\pi'(\pi^{-1}(u)) = v, \pi'(\pi^{-1}(v)) = u$ and $\pi'(w) = \pi(w)$ for $w \in V(H^c) \setminus \{ \pi^{-1}(u), \pi^{-1}(v)\}.$

    Thus, it is enough to prove that with high probability, $G\sim \mathcal{G}(n,1/2)$ contains a $\pi$-switch for every bijection $\pi \colon V(H^c) \rightarrow V(G)$. We show that if $n$ is large enough, then for any fixed $\pi$, $G$ contains a $\pi$-switch with probability at least $1-e^{-n\log n}$, and the statement then follows by a union bound over all $\pi$ since $n!\cdot e^{-n\log n} = o(1).$ Let us remark here that we do not need to assume that $\pi$ is an embedding of $H^c$ into $G$.

    Fix $\pi$ to be an arbitrary bijection from $V(H^c)$ to $V(G)$. Let $A\subseteq V(H^c)$ be the set of vertices of degree at least $4C$ in $H^c$ and set $B=V(H^c)\setminus A$. By assumption, $H^c$ has at most $Cn$ edges, implying $|A|\leq n/2$, and thus, $|B|\geq n/2$. Let $B'\subseteq B$ be a maximal independent subset of $B$ and note that $|B'|\geq \frac{n}{2(4C+1)}$. 
    \begin{claim} \label{claim:set-T}
        There exists an integer $D$ with $1\leq D\leq 6^{4C}$ and a set $T\subseteq B'$ with $|T|\geq n/\log^D(n)$ such that for every $v \in V(H) \setminus T$, we have $T \subseteq N_{H^c}(v)$ or $|N_{H_c}(v)\cap T| < n / \log^{6D} n$.
        
    \end{claim}
    \begin{proof}
        We will find the desired set $T$ iteratively. Set $T_0 = B'$ and note that $|T_0|\geq n/\log(n)$ since we may assume that $n$ is large enough. Then, we proceed in steps $i = 0, 1, \dots$ as follows. If there exists $v_i\in V(H) \setminus T_i$ such that $|N_{H_c}(v_i)\cap T_i|\geq n / \log^{6^{i+1}} (n)$, but $T_i\not\subseteq N_{H_c}(v_i)$, we set $T_{i+1}= N_{H_c}(v_i)\cap T_i$ and proceed to step $i+1$. If there exists no such $v_i$, then set $T = T_i$. Note that $|T_{i+1}| < |T_i|$ so the process eventually terminates. Let $\ell$ be the step in which we terminate, so that we set $T = T_\ell$. For each $0 \leq i\leq \ell$, by the definition of $T_i$, we have $|T_i|\geq  n / \log^{6^{i}}(n)$. Note that the claim immediately follows if $\ell \leq 4C$, which we now show. For each $0\leq i\leq \ell-1$, it holds that $T_\ell\subseteq T_{i+1}\subseteq N_{H^c}(v_i)$. Additionally, for $0\leq i<j\leq \ell-1$, we have $v_i \neq v_j,$ since $T_j\subseteq T_{i+1}\subseteq N_{H^c}(v_i)$ but $T_{j}\not\subseteq N_{H^c}(v_j)$. It follows that each vertex in $T_\ell$ contains $v_0,\ldots, v_{\ell-1}$ in its neighborhood in $H^c$. However, $T_\ell\subseteq B$ and therefore, each vertex in $T$ has at most $4C$ neighbors in $H^c$, implying $\ell \le 4C$ and thus completing the proof of the claim.
    \end{proof}
    Let $1\leq D\leq 6^{4C}$ and $T\subseteq B'$ be as provided by Claim~\ref{claim:set-T}. For every pair $\{u,v\}$, where $u,v\in T$ are distinct, let $X_{\{u,v\}}$ be the indicator variable for $\{\pi(u),\pi(v)\}$ being a $\pi$-switch and let $S$ be the sum of all these indicator random variables. Since each vertex in $T$ has at most $4C$ neighbors in $H^c$, we get that $\Pr[X_{\{u,v\}} = 1]\geq 2^{-8C}$ and hence, $\mathbb{E}[S]\geq 2^{-8C}\binom{|T|}{2}$. 

    Observe that $X_{\{u,v\}}$, $u,v\in T$, depends only on edges between $\pi(v)$ and $\pi(N_{H^c}(u)\setminus N_{H^c}(v))$ as well as edges between $\pi(u)$ and $\pi(N_{H^c}(v)\setminus N_{H^c}(u))$. For every $w\in V(H)$, let us denote by $d_w = |N_{H^c}(w) \cap T|$. For distinct $y,z \in V(G),$ we denote by $b_{yz}$ the number of terms of $S$ that depend on the presence of the edge $yz$ in $G$.
    
    Note that if $d_w \ge n / \log^{6D} n,$ then $T \subseteq N_{H^c}(w),$ so $b_{\pi(w)\pi(u)} = 0$ for all $u \in V(H^c) \setminus \{w\}.$ Furthermore, recall that $T\subseteq B'$ is an independent set. Therefore, $b_{\pi(u)\pi(v)} =0$ for any $u, v \in T$. If $u, v \not\in T,$ then also trivially $b_{\pi(u)\pi(v)} = 0$. Finally, if $w \not\in T,$ then for any $u \in T,$ presence of the edge $\pi(w)\pi(u)$ can only influence variables $X_{\{u,v\}}$ for $v \in N_{H^c}(w)$. Thus, for $w \not\in T, u \in T,$ we have $b_{\pi(u)\pi(w)} \le d_w.$ Denote by $R$ the set of vertices $w$ in $V(H) \setminus T$ with $d_w < n / \log^{6D} n.$ Observe that 
    \[ \sum_{w \in R} d_w^2 \le  n / \log^{6D} n \cdot \sum_{w \in R} d_w \le 2Cn^2 / \log^{6D} n. \]
    Putting these observations together, we have
    \begin{align*}    
        \sum_{y, z \in V(G)} b_{yz}^2 \leq |T| \cdot\sum_{w \in R}  d_w^2 \le |T| \cdot 2Cn^2/ \log^{6D} n \le \frac{2Cn^3}{ \log^{6D} n}.
    \end{align*}    
    By \Cref{thm: azumas}, it follows that
    \begin{align*}
    \Pr[S = 0] &\leq \exp\left(\frac{-2 (\E[X])^2}{\sum_{y, z \in V(G)} b_{yz}^2} \right) \le \exp\left( \frac{-\left(2^{-8C} \binom{|T|}{2} \right)^2 \cdot \log^{6D} n}{C n^3}\right)\\
    &\leq \exp\left(2^{-16C - 4} n \log^{2D} n / C\right) \le e^{-n \log n},
    \end{align*}
    finishing the proof.
\end{proof}
\begin{proof}[Proof of \Cref{thm: main}]
    For the sake of contradiction, suppose there is a constant $\delta > 0$ such that $f(n)\geq \delta$ for all large enough $n$. So assume that $n$ is large enough and let $H$ be a graph on $n$ vertices satisfying $f(H) \ge \delta$. Applying Lemma~\ref{lem:can-look-at-random-graph}, we obtain $\Pr_{G \sim \calG(n, 1/2)}[G \unqembed H] \ge \delta / 2.$ By Lemma~\ref{lem:constant-non-degree}, we have that $e(H) \ge \binom{n}{2} - Cn,$ for some constant $C = C(\delta)$. Finally, since $n$ is large enough, Lemma~\ref{lem: linear is not it} implies $\Pr_{G \sim \calG(n, 1/2)}[G \unqembed H] < \delta/2,$ a contradiction.    
\end{proof}

\section{Concluding remarks}
Though we have not checked it very thoroughly, we think a careful analysis of our argument shows that $f(n) = O\left( \frac{\log \log \log n}{\log \log n} \right)$. Since we do not believe this to be close to the truth, we have chosen to present the proof of $f(n) = o(1)$ for the sake of clarity of presentation.

Recall that we crucially used the anti-concentration of the number of edges in $\calG(n, 1/2)$, so in order to prove that $f(n) = o(1/n),$ new ideas are needed. Furthermore, it seems hard to push beyond\linebreak$f(n) = O(\log \log n / \log n)$ by considering a random graph process as we do. For $f(n) = \log \log n / \log n,$ the analogue of Lemma~\ref{lem:constant-non-degree} would only imply that $e(H) \ge \binom{n}{2} - O(n \log n)$. However, this does not seem to be particularly useful and, in fact, it can be deduced much more simply as otherwise in expectation $\calG(n, 1/2)$ is expected to have less than $n!\cdot e^{-n \log n}\ll 1$ embeddings into $H$.

\textbf{Acknowledgement.} We thank Noga Alon for helpful discussions about the problem.

\end{document}